\documentclass[a4paper,12pt,reqno]{amsart}
\usepackage{amsmath,amsthm,amssymb}
\usepackage[numbers,sort&compress]{natbib}
\nonstopmode \numberwithin{equation}{section}

\newtheorem{theorem}{Theorem}[section]
\newtheorem{proposition}{Proposition}[section]

\newtheorem{lemma}{Lemma}[section]

\newtheorem{remark}{Remark}[section]

\allowdisplaybreaks
\DeclareMathOperator{\real}{Re}

\allowdisplaybreaks

\begin{document}

\title[The generalized $\mathtt{k}$-Bessel functions]{Representation formulae and Monotonicity of the generalized $\mathtt{k}$-Bessel functions}
\author[Saiful R. Mondal]{Saiful R. Mondal}
\address{ Department of Mathematics and Statistics, College of Science,
King Faisal University, Al-Hasa 31982, Saudi Arabia}
\email{smondal@kfu.edu.sa}
\begin{abstract}
This paper introduces and studies a generalization of the $\mathtt{k}$-Bessel function of order $\nu$ given by
\[\mathtt{W}^{\mathtt{k}}_{\nu, c}(x):= \sum_{r=0}^\infty \frac{(-c)^r}{\Gamma_{\mathtt{k}}\left( r \mathtt{k} +\nu+\mathtt{k}\right) r!} \left(\frac{x}{2}\right)^{2r+\frac{\nu}{\mathtt{k}}}.
\]
Representation formulae are derived for $\mathtt{W}^{\mathtt{k}}_{\nu, c}.$ Further the function $\mathtt{W}^{\mathtt{k}}_{\nu, c}$ is shown to be a solution of a second order differential equation. Monotonicity and log-convexity properties for the generalized $\mathtt{k}$-Bessel function $\mathtt{W}^{\mathtt{k}}_{\nu, c}$  are investigated, particularly in the case $c=-1$. Several inequalities, including the Tur\'an-type inequality are established.

\end{abstract}

\maketitle

\section{Introductions}
Motivated with the repeated appearance of the expression
\[x(x + \mathtt{k})(x + 2\mathtt{k})\ldots(x + (n -1)\mathtt{k})\]
in  the combinatorics of creation and annihilation
operators \cite{Diaz2, Diaz3} and the perturbative computation of Feynman integrals,
see \cite{Deligne}, a generalization of the well-known Pochhammer symbols is given in \cite{Diaz} as
\[(x)_{n,\mathtt{k}}:=x(x + \mathtt{k})(x + 2\mathtt{k})\ldots(x + (n -1)\mathtt{k}),\]
 for all $\mathtt{k}>0$ and called it as  Pochhammer k-symbol. The closely associated functions relate with the Pochhammer symbols are the gamma and beta functions. Thus it is natural to introduce about $\mathtt{k}$-gamma and $\mathtt{k}$-beta function.
The $\mathtt{k}$-gamma  functions, denoted as $\Gamma_{\mathtt{k}}$,  is   studied in \cite{Diaz}, and defined  by
 \begin{align}
 \Gamma_{\mathtt{k}}(x):= \int_0^\infty t^{x-1}e^{-\frac{t^{\mathtt{k}}}{\mathtt{k}}} dt,
 \end{align}
for  $\real(x)>0$. Several properties of the $\mathtt{k}$-gamma functions and it's applications to generalize other related functions like as $\mathtt{k}$-beta functions, $\mathtt{k}$-digamma functions, can be seen in the articles \cite{Diaz, Kwara14, Mubeen13} and references  therein.

The  $k$-digamma functions defined by $\Psi_\mathtt{k}:=\Gamma_\mathtt{k}'/\Gamma_\mathtt{k} $ is  studied in \cite{Kwara14}. This functions have the series representation as
\begin{align}\label{def-digamma}
\Psi_\mathtt{k}(t):=\frac{\log(\mathtt{k})-\gamma_1}{\mathtt{k}}-\frac{1}{t}
+\sum_{n=1}^\infty \frac{t}{n\mathtt{k}(n\mathtt{k}+t)}
\end{align}
where $\gamma_1$ is the Euler-Mascheroni’s constant.

A calculation yields
\begin{align}\label{def-digamma-2} \Psi_\mathtt{k}'(t)=\sum_{n=0}^\infty  \frac{1}{(n\mathtt{k}+t)^2}, \quad \mathtt{k}>0 \quad \text{and} \quad t>0.\end{align}
Clearly, $\Psi_\mathtt{k}$ is increasing on $(0, \infty)$.

The Bessel function of order $p$ is given by
\begin{equation}\label{Strv}
\mathtt{J}_p(x):=  \sum_{k=0}^\infty \frac{(-1)^k}{\mathrm{\Gamma}{\left(k +p+1\right)} \mathrm{\Gamma}{\left(k+1\right)}} \left(\frac{x}{2}\right)^{2k+p}
\end{equation}
is a particular solution of the  Bessel differential equation
 \begin{align}\label{BesselDE}
 x^2 y''(x)+ x y'(x)+(x^2-p^2)y(x)= 0.
 \end{align}
 Here $\mathrm{\Gamma}$ denote the gamma function. A solution of the modified Bessel equation
 \begin{align}\label{modBesselDE}
 x^2 y''(x)+ x y'(x)-(x^2+{\nu}^2)y(x)= 0.
 \end{align}
 yields the modified Bessel function
 \begin{equation}\label{modBessel}
\mathtt{I}_\nu(x):=  \sum_{k=0}^\infty \frac{1}{\mathrm{\Gamma}{\left( k +\nu+1\right)} \mathrm{\Gamma}{\left( k+1\right)}} \left(\frac{x}{2}\right)^{2k+\nu}.
\end{equation}

The Bessel function has gone through several generalizations and investigations, notably in \cite{Baricz3,Galues-Bessel}. In \cite{Baricz3}, generalized Bessel function defined on the complex plane, and obtained sufficient conditions for it to be univalent, starlike, close-to-convex, and convex. This generalization is given by the power series
\begin{align}\label{gbf-Baricz}
\mathcal{W}_{p, b, c}(z)= \sum_{k=0}^\infty \frac{(-c)^k \left(\frac{z}{2}\right)^{2k+p+1}}{\mathrm{\Gamma}\left(k+\frac{3}{2}\right) \mathrm{\Gamma}\left(k+p+\frac{b+2}{2}\right)}, \quad \quad p, b, c \in \mathbb{C}.
\end{align}

In this article we will consider the function defined by the series
\begin{align}\label{eqn-1}
\mathtt{W}_{\nu, c}^{\mathtt{k}} (x) :=\sum_{r=0}^\infty \frac{(-c)^r }{ \Gamma_{\mathtt{k}}(r \mathtt{k}+ \nu+ \mathtt{k}) r!} \left(\frac{x}{2}\right)^{2r+\frac{\nu}{\mathtt{k}}}.
\end{align}
where $\mathtt{k}>0$, $\nu>-1$, and $c \in \mathbb{R}$. Since for $\mathtt{k}\to 1$, the $\mathtt{k}$-Bessel functions
$\mathtt{W}_{\nu, 1}^{1}$ reduce to the classical Bessel function $J_{\nu}$, while $\mathtt{W}_{\nu, -1}^{1}$ is equivalent to the modified Bessel function $I_\nu$. Thus, we call
 the function $\mathtt{W}_{\nu, c}^{\mathtt{k}}$  as the generalized $\mathtt{k}$-Bessel functions. The basic properties about the $\mathtt{k}$-Bessel functions can be seen in the work by Ghelot et. al. \cite{Gehlot-1,Gehlot-2,Gehlot-3}.

Tur\'an \cite{Turan} proved that the Legendre polynomials $P_n(x)$ satisfy the determinantal inequality
\begin{align}\label{turan-ine}
\left|
          \begin{array}{cc}
            P_n(x) & P_{n+1}(x) \\
            P_{n+1}(x) & P_{n+2}(x) \\
          \end{array}
      \right|
\leq  0, \quad -1 \le x \le 1
\end{align}
where $n = 0, 1, 2, \ldots$ and equality occurs only if $x = \pm 1$. The above classical result has been extended in several directions, for example, ultraspherical polynomials, Laguerre and Hermite polynomials,
Bessel functions of the first kind, modified Bessel functions, Polygamma etc.  Karlin and Szegö \cite{Karlin} named determinants such  in \eqref{turan-ine} as  Tur\'anians.

In Section $\ref{sec-2}$, representation formulae and few recurrence relation for $\mathtt{W}_{\nu, c}^{\mathtt{k}}$ will be derived.  More importantly, the function $\mathtt{W}_{\nu, c}^{\mathtt{k}}$ is shown to be a solution of a certain differential equation of second order, which reduces to \eqref{BesselDE} and \eqref{modBesselDE} in the case $\mathtt{k}=1$ and for particular values of $c.$ At the end of the Section $\ref{sec-2}$, two type integral representations of $\mathtt{W}_{\nu, c}^{\mathtt{k}}$ are also given.

Section $\ref{sec-3}$ is devoted to the investigation of monotonicity and log-convexity properties involving the function $\mathtt{W}_{\nu, c}^{\mathtt{k}},$ as well as  the ratio between two $\mathtt{k}$-Bessel functions of different order. As a consequence, Tur\'an-type inequalities are deduced.

\section{Representations for the $\mathtt{k}$-Bessel function}\label{sec-2}
\subsection{Differential equation} In this section we will find the differential equations corresponding to the functions $\mathtt{W}_{\nu, c}^{\mathtt{k}}$.

Differentiating both  side of \eqref{eqn-1} with respect to $x$, it follows that
\begin{align}\label{eqn-3}
x \frac{d}{dx}\mathtt{W}_{\nu, c}^{\mathtt{k}} (x) =\sum_{r=0}^\infty \frac{(-c)^r\left(2r+\frac{\nu}{\mathtt{k}}\right) }{ \Gamma_{\mathtt{k}}(r \mathtt{k}+ \nu+ \mathtt{k}) r!} \left(\frac{x}{2}\right)^{2r+\frac{\nu}{\mathtt{k}}}.
\end{align}
Recall that the $\mathtt{k}$-gamma function satisfy the relation $\Gamma_\mathtt{k}(z+\mathtt{k})= z \Gamma_\mathtt{k}(z)$. Now differentiate \eqref{eqn-3} with respect to $x$ and then using this property  yields
\begin{align*}
&x^2 \frac{d^2}{dx^2}\mathtt{W}_{\nu, c}^{\mathtt{k}} (x)+ x \frac{d}{dx}\mathtt{W}_{\nu, c}^{\mathtt{k}} (x)\\&=\sum_{r=0}^\infty \frac{(-c)^r\left(2r+\frac{\nu}{\mathtt{k}}\right)^2 }{ \Gamma_{\mathtt{k}}(r \mathtt{k}+ \nu+ \mathtt{k}) r!} \left(\frac{x}{2}\right)^{2r+\frac{\nu}{\mathtt{k}}}\\
&= \sum_{r=1}^\infty \frac{(-c)^r 4r\left(r+\frac{\nu}{\mathtt{k}}\right) }{ \Gamma_{\mathtt{k}}(r \mathtt{k}+ \nu+ \mathtt{k}) r!} \left(\frac{x}{2}\right)^{2r+\frac{\nu}{\mathtt{k}}}+
\frac{\nu^2}{\mathtt{k}^2} \sum_{r=0}^\infty \frac{(-c)^r }{ \Gamma_{\mathtt{k}}(r \mathtt{k}+ \nu+ \mathtt{k}) r!} \left(\frac{x}{2}\right)^{2r+\frac{\nu}{\mathtt{k}}}\\
&=\frac{4}{\mathtt{k}} \sum_{r=1}^\infty \frac{(-c)^r }{ \Gamma_{\mathtt{k}}(r \mathtt{k}+ \nu) (r-1)!} \left(\frac{x}{2}\right)^{2r+\frac{\nu}{\mathtt{k}}}+
\frac{\nu^2}{\mathtt{k}^2}\mathtt{W}_{\nu, c}^{\mathtt{k}} (x)\\
 &=- \frac{c x^2}{\mathtt{k}}\mathtt{W}_{\nu, c}^{\mathtt{k}} (x) +
\frac{\nu^2}{\mathtt{k}^2}\mathtt{W}_{\nu, c}^{\mathtt{k}} (x).
\end{align*}

Thus we have the following result.
\begin{proposition}Let $\mathtt{k}>0$ and $\nu >-k$. Then the function $\mathtt{W}_{\nu, c}^{\mathtt{k}}$ is the solution of the homogeneous  differential equation
\begin{align}\label{eqn-4}
\frac{d^2y}{dx^2}+ x^{-1} \frac{dy}{dx}+\frac{1}{\mathtt{k}^2}\left(c\; \mathtt{k}-
\frac{\nu^2}{x^2}\right)y=0.
\end{align}
\end{proposition}
\subsection{Recurrence relations}
From  \eqref{eqn-3}, we have
\begin{align*}
x \frac{d}{dx}\mathtt{W}_{\nu, c}^{\mathtt{k}} (x) &=\frac{1}{\mathtt{k}}\sum_{r=0}^\infty \frac{(-c)^r\left(2r \mathtt{k} +{\nu}\right) }{ \Gamma_{\mathtt{k}}(r \mathtt{k}+ \nu+ \mathtt{k}) r!} \left(\frac{x}{2}\right)^{2r+\frac{\nu}{\mathtt{k}}}\\
&=\frac{\nu}{\mathtt{k}}\sum_{r=0}^\infty \frac{(-c)^r\left(2r \mathtt{k} +{\nu}\right) }{ \Gamma_{\mathtt{k}}(r \mathtt{k}+ \nu+ \mathtt{k}) r!} \left(\frac{x}{2}\right)^{2r+\frac{\nu}{\mathtt{k}}}
+2 \sum_{r=1}^\infty \frac{(-c)^r}{ \Gamma_{\mathtt{k}}(r \mathtt{k}+ \nu+ \mathtt{k}) (r-1)!} \left(\frac{x}{2}\right)^{2r+\frac{\nu}{\mathtt{k}}}\\
&=\frac{\nu}{\mathtt{k}}\mathtt{W}_{\nu, c}^{\mathtt{k}}(x)
+2 \sum_{r=0}^\infty \frac{(-c)^{r+1}}{ \Gamma_{\mathtt{k}}(r \mathtt{k}+ \nu+ 2\mathtt{k}) r!} \left(\frac{x}{2}\right)^{2r+2+\frac{\nu}{\mathtt{k}}}\\
&=\frac{\nu}{\mathtt{k}}\mathtt{W}_{\nu, c}^{\mathtt{k}}(x)- xc \mathtt{W}_{\nu+\mathtt{k}, c}^{\mathtt{k}}(x).
\end{align*}
Thus, we have the difference equation
\begin{align}\label{rr-1}
x \frac{d}{dx}\mathtt{W}_{\nu, c}^{\mathtt{k}} (x)=\frac{\nu}{\mathtt{k}}\mathtt{W}_{\nu, c}^{\mathtt{k}}(x)- xc \mathtt{W}_{\nu+\mathtt{k}, c}^{\mathtt{k}}(x).
\end{align}
Again rewrite \eqref{eqn-3} as
\begin{align*}
x \frac{d}{dx}\mathtt{W}_{\nu, c}^{\mathtt{k}} (x) &=\frac{1}{\mathtt{k}}\sum_{r=0}^\infty \frac{(-c)^r\left(2r \mathtt{k} +2{\nu}\right)- \nu }{ \Gamma_{\mathtt{k}}(r \mathtt{k}+ \nu+ \mathtt{k}) r!} \left(\frac{x}{2}\right)^{2r+\frac{\nu}{\mathtt{k}}}\\
&=- \frac{\nu}{\mathtt{k}}\sum_{r=0}^\infty \frac{(-c)^r}{ \Gamma_{\mathtt{k}}(r \mathtt{k}+ \nu+ \mathtt{k}) r!} \left(\frac{x}{2}\right)^{2r+\frac{\nu}{\mathtt{k}}}
+2\sum_{r=0}^\infty \frac{(-c)^r\left(r \mathtt{k} +{\nu}\right)}{ \Gamma_{\mathtt{k}}(r \mathtt{k}+ \nu+ \mathtt{k}) r!} \left(\frac{x}{2}\right)^{2r+\frac{\nu}{\mathtt{k}}}\\
&=- \frac{\nu}{\mathtt{k}}  \mathtt{W}_{\nu, c}^{\mathtt{k}} (x)
+ \frac{x}{\mathtt{k}} \sum_{r=0}^\infty \frac{(-c)^r}{ \Gamma_{\mathtt{k}}(r \mathtt{k}+ \nu- \mathtt{k} + \mathtt{k}) r!} \left(\frac{x}{2}\right)^{2r+\frac{\nu-\mathtt{k}}{\mathtt{k}}}\\
&=- \frac{\nu}{\mathtt{k}}  \mathtt{W}_{\nu, c}^{\mathtt{k}} (x)
+ \frac{x}{\mathtt{k}} \mathtt{W}_{\nu-\mathtt{k}, c}^{\mathtt{k}} (x).
\end{align*}
This give us the second difference equation as
\begin{align}\label{rr-2}
x \frac{d}{dx}\mathtt{W}_{\nu, c}^{\mathtt{k}} (x)
=\frac{x}{\mathtt{k}} \mathtt{W}_{\nu-\mathtt{k}, c}^{\mathtt{k}} (x)
- \frac{\nu}{\mathtt{k}}  \mathtt{W}_{\nu, c}^{\mathtt{k}} (x).
\end{align}
Thus \eqref{rr-1} and \eqref{rr-2} leads to the following recurrence relations.
\begin{proposition}
Let $\mathtt{k}>0$ and $\nu > -\mathtt{k}$. Then
\begin{align}\label{rr-3}
2 \nu\mathtt{W}_{\nu, c}^{\mathtt{k}}(x)
&=x  \mathtt{W}_{\nu-\mathtt{k}, c}^{\mathtt{k}} (x)
+ xc \mathtt{k} \mathtt{W}_{\nu+\mathtt{k}, c}^{\mathtt{k}}(x),\\
\label{rr-7}
\mathtt{W}^\mathtt{k}_{\nu-\mathtt{k}, c}(x)&=\frac{2}{x} \sum_{r=0}^\infty (-1)^r (\nu+ 2 r \mathtt{k}) \mathtt{W}^{\mathtt{k}}_{\nu+ 2r \mathtt{k}, c}(x)\\ \label{rr-5}
\frac{d}{dx} \left( x^{\frac{\nu}{\mathtt{k}}} \mathtt{W}^\mathtt{k}_{\nu, c}(x)\right)&=\frac{x^{\frac{\nu}{\mathtt{k}}}}{\mathtt{k}} \mathtt{W}^\mathtt{k}_{\nu-\mathtt{k}, c}(x)\\ \label{rr-6}
\frac{d}{dx} \left( x^{-\frac{\nu}{\mathtt{k}}} \mathtt{W}^\mathtt{k}_{\nu, c}(x)\right)&=-c x^{-\frac{\nu}{\mathtt{k}}} \mathtt{W}^\mathtt{k}_{\nu+\mathtt{k}, c}(x)\\  \label{rr-8}
 \frac{d^m}{dx^m}\left( \mathtt{W}^{\mathtt{k}}_{\nu, c}(x)\right)
&= \frac{1}{2^m \mathtt{k}^m} \sum_{n=0}^m (-1)^n
\left(
  \begin{array}{c}
    m \\
    n \\
  \end{array}
\right) c^n \mathtt{k}^n \mathtt{W}^{\mathtt{k}}_{\nu-m \mathtt{k}+2 n \mathtt{k}, c}(x) \quad \text{for all} \quad m \in \mathbb{N}.
\end{align}
\end{proposition}

\begin{proof}
The relation \eqref{rr-3} follows by subtracting \eqref{rr-2} from \eqref{rr-1}.

Next to establish \eqref{rr-7}, lets rewrite \eqref{rr-3} as
\begin{align}\label{rr-9}
 \mathtt{W}_{\nu-\mathtt{k}, c}^{\mathtt{k}} (x)
+ c \mathtt{k} \mathtt{W}_{\nu+\mathtt{k}, c}^{\mathtt{k}}(x)
= 2 \frac{\nu}{x}\mathtt{W}_{\nu, c}^{\mathtt{k}}(x)
\end{align}
Now multiply both side of \eqref{rr-9} by $-c \mathtt{k}$ and replace $\nu$ by $\nu+2\mathtt{k}$. Then we have
\begin{align}
- c \mathtt{k} \mathtt{W}_{\nu+\mathtt{k}, c}^{\mathtt{k}} (x)
-c^2 \mathtt{k}^2 \mathtt{W}_{\nu+3 \mathtt{k}, c}^{\mathtt{k}}(x)
= -2 c\mathtt{k} \frac{\nu+2 \mathtt{k}}{x}\mathtt{W}_{\nu+2 \mathtt{k}, c}^{\mathtt{k}}(x).
\end{align}
Similarly, multiplying  both side of \eqref{rr-9} by $c^2 \mathtt{k}^2$ and replacing $\nu$ by $\nu+4\mathtt{k}$, we get
\begin{align}
 c^2 \mathtt{k}^2 \mathtt{W}_{\nu+3 \mathtt{k}, c}^{\mathtt{k}} (x)
+c^3 \mathtt{k}^3 \mathtt{W}_{\nu+5 \mathtt{k}, c}^{\mathtt{k}}(x)
= 2 c^2\mathtt{k}^2 \frac{\nu+4 \mathtt{k}}{x}\mathtt{W}_{\nu+4 \mathtt{k}, c}^{\mathtt{k}}(x).
\end{align}
If we continue like the above and addition them  leads to \eqref{rr-7}.

From the definition \eqref{eqn-1} it is clear that
\begin{align}\label{rr-10}
x^{\frac{\nu}{\mathtt{k}}}\mathtt{W}_{\nu, c}^{\mathtt{k}} (x) =\sum_{r=0}^\infty \frac{(-c)^r }{ \Gamma_{\mathtt{k}}(r \mathtt{k}+ \nu+ \mathtt{k})2^{2r+\frac{\nu}{\mathtt{k}}} r!} \left(x\right)^{2r+\frac{2\nu}{\mathtt{k}}}.
\end{align}
The derivative of \eqref{rr-10} with respect to $x$ yields
\begin{align*}
\frac{d}{dx}\left(x^{\frac{\nu}{\mathtt{k}}}\mathtt{W}_{\nu, c}^{\mathtt{k}} (x) \right) &=\sum_{r=0}^\infty \frac{(-c)^r (2r+\frac{2\nu}{\mathtt{k}}) }{ \Gamma_{\mathtt{k}}(r \mathtt{k}+ \nu+ \mathtt{k})2^{2r+\frac{\nu}{\mathtt{k}}} r!} \left(x\right)^{2r+\frac{2\nu}{\mathtt{k}}-1}\\
&=\frac{x^\frac{\nu}{\mathtt{k}}}{\mathtt{k}}\sum_{r=0}^\infty \frac{(-c)^r }{ \Gamma_{\mathtt{k}}(r \mathtt{k}+ \nu) r!} \left(\frac{x}{2}\right)^{2r+\frac{\nu}{\mathtt{k}}-1}
=\frac{x^\frac{\nu}{\mathtt{k}}}{\mathtt{k}}\mathtt{W}_{\nu-\mathtt{k}, c}^{\mathtt{k}} (x).
\end{align*}
Similarly,
\begin{align*}
\frac{d}{dx}\left(x^{-\frac{\nu}{\mathtt{k}}}\mathtt{W}_{\nu, c}^{\mathtt{k}} (x) \right) &=\sum_{r=1}^\infty \frac{(-c)^r 2r }{ \Gamma_{\mathtt{k}}(r \mathtt{k}+ \nu+ \mathtt{k})2^{2r+\frac{\nu}{\mathtt{k}}} r!} \left(x\right)^{2r-1}\\
&=x^{-\frac{\nu}{\mathtt{k}}}\sum_{r=1}^\infty \frac{(-c)^r }{ \Gamma_{\mathtt{k}}(r \mathtt{k}+ \nu+\mathtt{k}) (r-1)!} \left(\frac{x}{2}\right)^{2r+\frac{\nu}{\mathtt{k}}-1}\\
&=x^{-\frac{\nu}{\mathtt{k}}}\sum_{r=0}^\infty \frac{(-c)^{r+1} }{ \Gamma_{\mathtt{k}}(r \mathtt{k}+ \nu+2\mathtt{k}) r!} \left(\frac{x}{2}\right)^{2r+\frac{\nu}{\mathtt{k}}+1}
=-c x^{-\frac{\nu}{\mathtt{k}}} \mathtt{W}_{\nu+\mathtt{k}, c}^{\mathtt{k}} (x).
\end{align*}

The identity \eqref{rr-8} can be proved by using mathematical induction on $m$. Recall that
\[ \left(
  \begin{array}{c}
    r \\
    r \\
  \end{array}
\right)=\left(
  \begin{array}{c}
    r \\
    0 \\
  \end{array}
\right)=1
\quad \text{and} \quad \left(
  \begin{array}{c}
    r \\
    n \\
  \end{array}
\right)+\left(
  \begin{array}{c}
    r \\
    n-1 \\
  \end{array}
\right)= \left(
  \begin{array}{c}
    r+1 \\
    n \\
  \end{array}
\right)
\]
For $m=1$, the proof of the identity \eqref{rr-8} is equivalent to show that
\begin{align}\label{rr-4}
2 \mathtt{k} \frac{d}{dx}\mathtt{W}_{\nu, c}^{\mathtt{k}} (x)
&= \mathtt{W}_{\nu-\mathtt{k}, c}^{\mathtt{k}} (x)
- c \mathtt{k} \mathtt{W}_{\nu+\mathtt{k}, c}^{\mathtt{k}}(x).
\end{align}
This above relation can be obtained by simple adding \eqref{rr-1} and \eqref{rr-2}. Thus, identity \eqref{rr-8} hold for $m=1$.

Assume that the identity \eqref{rr-8} also holds for any $m=r \geq 2$, i.e.
\begin{align*}
 \frac{d^r}{dx^r}\left( \mathtt{W}^{\mathtt{k}}_{\nu, c}(x)\right)
&= \frac{1}{2^m \mathtt{k}^r} \sum_{n=0}^r (-1)^n
\left(
  \begin{array}{c}
    r \\
    n \\
  \end{array}
\right) c^n \mathtt{k}^n \mathtt{W}^{\mathtt{k}}_{\nu-r \mathtt{k}+2 n \mathtt{k}, c}(x).
\end{align*}
This implies for $m=r+1$,
\begin{align*}
 \frac{d^{r+1}}{dx^{r+1}}\left( \mathtt{W}^{\mathtt{k}}_{\nu, c}(x)\right)
&= \frac{1}{2^r \mathtt{k}^r} \sum_{n=0}^r (-1)^n
\left(
  \begin{array}{c}
    r \\
   n \\
 \end{array}
\right) c^n \mathtt{k}^n \frac{d}{dr}\mathtt{W}^{\mathtt{k}}_{\nu-r \mathtt{k}+2 n \mathtt{k}, c}(x).\\
&= \frac{1}{2^{r+1} \mathtt{k}^{r+1}} \sum_{n=0}^r (-1)^n
\left(
  \begin{array}{c}
    r \\
    n \\
  \end{array}
\right) c^n \mathtt{k}^n \bigg(\mathtt{W}^{\mathtt{k}}_{\nu-(r+1) \mathtt{k}+2 n \mathtt{k}, c}(x)-c \mathtt{k}\mathtt{W}^{\mathtt{k}}_{\nu-(r-1) \mathtt{k}+2 n \mathtt{k}, c}(x)\bigg).\\
&= \frac{1}{2^{r+1} \mathtt{k}^{r+1}} \sum_{n=0}^r (-1)^n
\left(
  \begin{array}{c}
    r \\
    n \\
  \end{array}
\right) c^n \mathtt{k}^n \mathtt{W}^{\mathtt{k}}_{\nu-(r+1) \mathtt{k}+2 n \mathtt{k}, c}(x)\\
& \quad\quad \quad  -
\frac{1}{2^{r+1} \mathtt{k}^{r+1}} \sum_{n=0}^r (-1)^n
\left(
  \begin{array}{c}
    r \\
    n \\
  \end{array}
\right) c^{n+1} \mathtt{k}^{n+1}\mathtt{W}^{\mathtt{k}}_{\nu-(r-1) \mathtt{k}+2 n \mathtt{k}, c}(x).\\
&=\frac{1}{2^{r+1} \mathtt{k}^{r+1}}
 \bigg[ \mathtt{W}^{\mathtt{k}}_{\nu-(r+1) \mathtt{k}, c}(x)
 +\sum_{n=1}^r (-1)^r \left(\left(
  \begin{array}{c}
    r \\
    n \\
  \end{array}
\right)+\left(
  \begin{array}{c}
    r \\
    n-1 \\
  \end{array}
\right)\right)\mathtt{W}^{\mathtt{k}}_{\nu-(r+1) \mathtt{k}+2n \mathtt{k}, c}(x)
 \bigg.\\
 &\quad \quad \quad \quad \quad \quad \quad\bigg.-(-1)^r c^{r+1}\mathtt{k}^{r+1}\mathtt{W}^{\mathtt{k}}_{\nu+(r+1) \mathtt{k}, c}(x)\bigg]\\
&=\frac{1}{2^{r+1} \mathtt{k}^{r+1}}
 \bigg[ \left(
  \begin{array}{c}
    r+1 \\
    0 \\
  \end{array}
\right)\mathtt{W}^{\mathtt{k}}_{\nu-(r+1) \mathtt{k}, c}(x)
 +\sum_{n=1}^r (-1)^r \left(
  \begin{array}{c}
    r+1 \\
    n \\
  \end{array}
\right)\mathtt{W}^{\mathtt{k}}_{\nu-(r+1) \mathtt{k}+2n \mathtt{k}, c}(x)\bigg.\\
 &\quad \quad \quad \quad\quad \quad\bigg.+(-1)^{r+1} \left(
  \begin{array}{c}
    r+1 \\
    r+1 \\
  \end{array}
\right) c^{r+1}\mathtt{k}^{r+1}\mathtt{W}^{\mathtt{k}}_{\nu-(r+1) \mathtt{k}+2(r+1) \mathtt{k}, c}(x)\bigg]\\
&=\frac{1}{2^{r+1} \mathtt{k}^{r+1}}
 \sum_{n=0}^{r+1} (-1)^r \left(
  \begin{array}{c}
    r+1 \\
    n \\
  \end{array}
\right)\mathtt{W}^{\mathtt{k}}_{\nu-(r+1) \mathtt{k}+2n \mathtt{k}, c}(x).
 \end{align*}
 Hence the conclusion by mathematical induction on $m$.
\end{proof}

\subsection{Integral representation}
Now we will derive two type integral representation of the functions $\mathtt{W}_{\nu, c}^{\mathtt{k}}$. For this purpose we need to recall  the $\mathtt{k}$-Beta functions from \cite{Diaz}. The $\mathtt{k}$ version of the beta functions is defined by
\begin{align}\label{eqn-6}
\mathtt{B}_{\mathtt{k}}(x, y)=\frac{ \Gamma_{\mathtt{k}}(x) \Gamma_{\mathtt{k}}(y)}{\Gamma_{\mathtt{k}}(x+y)}=\frac{1}{\mathtt{k}} \int_0^1 t^{\frac{x}{\mathtt{k}}-1}(1-t)^{\frac{y}{\mathtt{k}}-1} dt.
\end{align}
Substituting $t$ by $t^2$  on the right hand integration in \eqref{eqn-6}, it follows that
\begin{align}\label{eqn-7}
\mathtt{B}_{\mathtt{k}}(x, y)=\frac{2}{\mathtt{k}} \int_0^1 t^{\frac{2x}{\mathtt{k}}-1}(1-t^2)^{\frac{y}{\mathtt{k}}-1} dt.
\end{align}
Let, $x=(r+1) \mathtt{k}$ and $y=\nu$. Then from \eqref{eqn-6} and \eqref{eqn-7}, we have
\begin{align}\label{eqn-8}
\frac{ 1}{\Gamma_{\mathtt{k}}(r \mathtt{k}+\nu+\mathtt{k})}=\frac{2}{\Gamma_{\mathtt{k}}((r+1) \mathtt{k}) \Gamma_{\mathtt{k}}(\nu)} \int_0^1 t^{2r+1}(1-t^2)^{\frac{\nu}{\mathtt{k}}-1} dt.
\end{align}
Owing to \cite{Diaz}, we have the identity $\Gamma_{\mathtt{k}}(\mathtt{k} x)= \mathtt{k}^{x-1} \Gamma(x)$. This gives
\begin{align}\label{eqn-9}
\frac{ 1}{\Gamma_{\mathtt{k}}(r \mathtt{k}+\nu+\mathtt{k})}=\frac{2}{\mathtt{k}^{r}\Gamma(r+1) \Gamma_{\mathtt{k}}(\nu)} \int_0^1 t^{2r+1}(1-t^2)^{\frac{\nu}{\mathtt{k}}-1} dt.
\end{align}
Now \eqref{eqn-1} and \eqref{eqn-9} together yield
\begin{align}\label{eqn-10}\notag
\mathtt{W}_{\nu, c}^{\mathtt{k}} (x)& =\frac{2}{\Gamma_{\mathtt{k}}(\nu)} \left(\frac{x }{2}\right)^{\frac{\nu}{\mathtt{k}}}\int_0^1  t (1-t^2)^{\frac{\nu}{\mathtt{k}}-1} \sum_{r=0}^\infty \frac{(-c)^r }{ \Gamma(r+1)   r!} \left(\frac{x t}{2 \sqrt{\mathtt{k}}}\right)^{2r} dt.\\
&=\frac{2}{\Gamma_{\mathtt{k}}(\nu)} \left(\frac{x }{2}\right)^{\frac{\nu}{\mathtt{k}}}\int_0^1  t (1-t^2)^{\frac{\nu}{\mathtt{k}}-1} \mathcal{W}_{0, 1, c} \left(\frac{x t}{\sqrt{\mathtt{k}}}\right) dt,
\end{align}
where $\mathcal{W}_{p, b, c}$  is defined in \eqref{gbf-Baricz}.

Now for the second integral representation, substitute $x= r+\mathtt{k}/2$ and
$y= \nu+\mathtt{k}/2$ in \eqref{eqn-7}. Then, \eqref{eqn-8} can be rewrite as
\begin{align}\label{eqn-11}
\frac{ 1}{\Gamma_{\mathtt{k}}(r \mathtt{k}+\nu+\mathtt{k})}=\frac{2}{\Gamma_{\mathtt{k}}\left(\left(r+\frac{1}{2}\right) \mathtt{k}\right) \Gamma_{\mathtt{k}}\left(\nu+\frac{\mathtt{k}}{2}\right) } \int_0^1 t^{2r}(1-t^2)^{\frac{\nu}{\mathtt{k}}-\frac{1}{2}} dt.
\end{align}
Again the identity $\Gamma_{\mathtt{k}}(\mathtt{k} x)= \mathtt{k}^{x-1} \Gamma(x)$ yields
\begin{align}\label{eqn-12}
\Gamma_{\mathtt{k}}\left(\left(r+\frac{1}{2}\right) \mathtt{k}\right)
=\mathtt{k}^{r-\frac{1}{2}}\Gamma\left(r+\frac{1}{2}\right).
\end{align}
Further, the Legendre  duplication formula (see \cite{ABRAMOWITZ,Andrews-Askey})
\begin{equation}\label{eqn:Legendre-duplication}
\mathrm{\Gamma}{(z)}\mathrm{\Gamma}{\left(z+\tfrac{1}{2}\right)}= 2^{1-2z}\; \sqrt{\pi}\; \mathrm{\Gamma}{(2z)}
\end{equation}
shows that
\[\Gamma\left(r+\frac{1}{2}\right) r!= r \Gamma\left(r+\frac{1}{2}\right)\Gamma(r)= \frac{\sqrt{\pi} (2r)!}{2^{2r}}.\]
This along with \eqref{eqn-11} and \eqref{eqn-12} reduce the series of $\mathtt{W}^{\mathtt{k}}_{\nu, c}$ as
\begin{align}\label{eqn-13}\notag
\mathtt{W}_{\nu, c}^{\mathtt{k}} (x)& =\frac{2\sqrt{\mathtt{k}}}{\Gamma_{\mathtt{k}}\left(\nu+\frac{\mathtt{k}}{2}\right)} \left(\frac{x }{2}\right)^{\frac{\nu}{\mathtt{k}}}\int_0^1   (1-t^2)^{\frac{\nu}{\mathtt{k}}-\frac{1}{2}} \sum_{r=0}^\infty \frac{(-c)^r }{ \Gamma(r+1)   r!} \left(\frac{x t}{2 \sqrt{\mathtt{k}}}\right)^{2r} dt.\\
&=\frac{2\sqrt{\mathtt{k}}}{\sqrt{\pi}\Gamma_{\mathtt{k}}\left(\nu+\frac{\mathtt{k}}{2}\right)} \left(\frac{x }{2}\right)^{\frac{\nu}{\mathtt{k}}}\int_0^1   (1-t^2)^{\frac{\nu}{\mathtt{k}}-\frac{1}{2}} \sum_{r=0}^\infty \frac{(-c)^r }{ (2r)!} \left(\frac{x t}{ \sqrt{\mathtt{k}}}\right)^{2r} dt.
\end{align}
Finally for $c=\pm \alpha^2$, $\alpha\in \mathbb{R}$, the representation \eqref{eqn-13} respectively leads to
\begin{align}\label{int-14}
\mathtt{W}_{\nu, \alpha^2}^{\mathtt{k}} (x) =\frac{2\sqrt{\mathtt{k}}}{\sqrt{\pi}\Gamma_{\mathtt{k}}\left(\nu+\frac{\mathtt{k}}{2}\right)} \left(\frac{x }{2}\right)^{\frac{\nu}{\mathtt{k}}}\int_0^1   (1-t^2)^{\frac{\nu}{\mathtt{k}}-\frac{1}{2}} \cos\left(\frac{\alpha x t}{ \sqrt{\mathtt{k}}}\right) dt.
\end{align}
and
\begin{align}\label{int-15}
\mathtt{W}_{\nu, -\alpha^2}^{\mathtt{k}} (x) =\frac{2\sqrt{\mathtt{k}}}{\sqrt{\pi}\Gamma_{\mathtt{k}}\left(\nu+\frac{\mathtt{k}}{2}\right)} \left(\frac{x }{2}\right)^{\frac{\nu}{\mathtt{k}}}\int_0^1   (1-t^2)^{\frac{\nu}{\mathtt{k}}-\frac{1}{2}} \cosh\left(\frac{\alpha x t}{ \sqrt{\mathtt{k}}}\right) dt.
\end{align}
If $\nu=\mathtt{k}/2$, then from \eqref{int-14} a computation give the relation between sine functions and generalized $\mathtt{k}$-Bessel functions by
\[\sin\left(\frac{\alpha x}{\sqrt{\mathtt{k}}}\right)= \frac{\alpha}{\mathtt{k}}\sqrt{\frac{\pi x}{2}}\mathtt{W}_{\frac{\nu}{\mathtt{k}}, \alpha^2}^{\mathtt{k}} (x).\]
Similarly, the relation
\[\sinh\left(\frac{\alpha x}{\sqrt{\mathtt{k}}}\right)= \frac{\alpha}{\mathtt{k}}\sqrt{\frac{\pi x}{2}}\mathtt{W}_{\frac{\nu}{\mathtt{k}}, -\alpha^2}^{\mathtt{k}} (x)\]
 can be derive from \eqref{int-15}.

\section{Monotonicity and log-convexity properties}\label{sec-3}  This section is devoted for the
discussion of the monotonicity and the log-convexity properties of the modified $\mathtt{k}$-Bessel function $\mathtt{W}_{\nu, -1}^{\mathtt{k}}=\mathtt{I}_\nu^{\mathtt{k}}$. As  consequences of those results, we derive several functional inequalities for $\mathtt{I}_\nu^{\mathtt{k}}$.

The following result of Biernacki and Krzy\.z \cite{Biernacki-Krzy} will be required.
\begin{lemma}\label{lemma:1}\cite{Biernacki-Krzy}
Consider the power series $f(x)=\sum_{k=0}^\infty a_k x^k$ and $g(x)=\sum_{k=0}^\infty b_k x^k$, where $a_k \in \mathbb{R}$ and $b_k > 0$ for all $k$. Further suppose that both series converge on $|x|<r$. If the sequence $\{a_k/b_k\}_{k\geq 0}$ is increasing (or decreasing), then the function $x \mapsto f(x)/g(x)$ is also increasing (or decreasing) on $(0,r)$.
\end{lemma}
The above lemma still holds when both $f$ and  $g$ are even, or both are odd functions.\\

We now state and prove our main results in this section.
Consider the functions
\begin{align}\label{normalized-k}
\mathcal{I}_\nu^{\mathtt{k}}(x):= \left(\frac{2}{x}\right)^{\frac{\nu}{\mathtt{k}}}\Gamma_{\mathtt{k}}
     (\nu+\mathtt{k}) \mathtt{I}_\nu^{\mathtt{k}}(x)=\sum_{r=0}^\infty f_r(\nu) x^{2r},
\end{align}
where,
\begin{align}\label{coef-1}
f_r(\nu)= \frac{\Gamma_{\mathtt{k}}{(\nu+\mathtt{k})}}{\Gamma_{\mathtt{k}}{(r{\mathtt{k}}+\nu+\mathtt{k})}4^r r!}.
\end{align}

Then we have the following properties.

\begin{theorem}\label{finalThm} Let $\mathtt{k}>0$. Following results are true for the modified $\mathtt{k}$-Bessel functions.
\begin{enumerate}
\item[(a)] If $\nu \geq \mu >-\mathtt{k}$, then the function $x \mapsto {\mathcal{I}_\mu^{\mathtt{k}}(x)}/
    {\mathcal{I}_\nu^{\mathtt{k}}(x)}$ is increasing on $\mathbb{R}$.
\item[(b)] The function $\nu \mapsto  \mathcal{I}^{\mathtt{k}}_{\nu+\mathtt{k}}(x)
/\mathcal{I}^{\mathtt{k}}_{\nu}(x)$ is increasing on $(-\mathtt{k}, \infty)$
 that is, for $\nu \geq \mu >-\mathtt{k}$, the inequality
 \begin{equation}\label{eqn:thm-3}
\mathcal{I}_{\nu+\mathtt{k}}^{\mathtt{k}}(x)\mathcal{I}_{\mu}^{\mathtt{k}}(x)\geq  \mathcal{I}_{\nu}^{\mathtt{k}}(x)\mathcal{I}_{\mu+\mathtt{k}}^{\mathtt{k}}(x)
\end{equation}
holds for each fixed $x>0$ and $\mathtt{k}>0$.
\item[(c)] The function $\nu \mapsto \mathcal{I}_\nu^{\mathtt{k}}(x)$ is decreasing and log-convex on $(-\mathtt{k}, \infty)$ for each fixed $x >0$.
\end{enumerate}
\end{theorem}
\begin{proof}
(a). From \eqref{normalized-k} it follows that
\begin{align*}
\frac{\mathcal{I}^{\mathtt{k}}_{\nu}(x)}{\mathcal{I}^{\mathtt{k}}_{\mu}(x)}
=\frac{\sum_{r=0}^\infty f_r(\nu) x^{2r}}{\sum_{r=0}^\infty f_r(\mu) x^{2r}}.
\end{align*}
Denote $w_r:=f_r(\nu)/f_r(\mu)$. Then
\[w_r= \frac{\Gamma_{\mathtt{k}}{(\nu+\mathtt{k})}\Gamma_{\mathtt{k}}
{(r{\mathtt{k}}+\mu+\mathtt{k})}}
{\Gamma_{\mathtt{k}}{(\mu+\mathtt{k})}
\Gamma_{\mathtt{k}}{(r{\mathtt{k}}+\nu+\mathtt{k})}}.\]
Now using the properties $\Gamma_{\mathtt{k}}{(y+\mathtt{k})}=y\Gamma_{\mathtt{k}}{(y)}$
it can be shown that
\begin{align*}
\frac{w_{r+1}}{w_r}&=
\frac{\Gamma_{\mathtt{k}}{(r{\mathtt{k}}+\nu+\mathtt{k})}\Gamma_{\mathtt{k}}
{(r{\mathtt{k}}+\mu+2\mathtt{k})}}
{\Gamma_{\mathtt{k}}{(r{\mathtt{k}}+\mu+\mathtt{k})}
\Gamma_{\mathtt{k}}{(r{\mathtt{k}}+\nu+2\mathtt{k})}}
=\frac{r{\mathtt{k}}+\mu+\mathtt{k}}{r{\mathtt{k}}+\nu+\mathtt{k}}\leq 1,
\end{align*}
in view of $\nu \geq \mu >-\mathtt{k}.$  Now the result follows from the Lemma \ref{lemma:1}. \\

(b). Let $\nu \geq \mu>-\mathtt{k}.$ It follows from part $(a)$ that
\begin{align*}
\frac {d}{dx} \left(\frac{\mathcal{I}^{\mathtt{k}}_{\nu}(x)}{\mathcal{I}^{\mathtt{k}}_{\mu}(x)}\right)\geq 0
\end{align*}
on $(0,\infty).$ Thus
\begin{align}\label{eqn:xxx}
\left(\mathcal{I}^{\mathtt{k}}_{\nu}(x)\right)'
\left(\mathcal{I}^{\mathtt{k}}_{\mu}(x)\right)-
\left(\mathcal{I}^{\mathtt{k}}_{\nu}(x)\right)
\left(\mathcal{I}^{\mathtt{k}}_{\mu}(x)\right)'\geq 0.
\end{align}
It now follows from $(\ref{rr-6})$ that
\begin{align*}
 \frac{x}{2}\bigg(\mathcal{I}^{\mathtt{k}}_{\nu+k}(x)\; \mathcal{I}^{\mathtt{k}}_{\mu}(x)-\mathcal{I}^{\mathtt{k}}_{\mu+k}(x)\; \mathcal{I}^{\mathtt{k}}_{\nu}(x)\bigg) \geq 0,
\end{align*}
whence $\mathcal{I}^{\mathtt{k}}_{\nu+k}/\mathcal{I}^{\mathtt{k}}_{\nu}$ is increasing for $\nu>-\mathtt{k}$ and for some fixed $x >0$.\\

(c). It is clear that for all $\nu >-\mathtt{k}$,
\[f_r(\nu)= \frac{\Gamma_{\mathtt{k}}{(\nu+\mathtt{k})}}
{\Gamma_{\mathtt{k}}{(r{\mathtt{k}}+\nu+\mathtt{k})}4^r r!}>0.\]
A logarithmic differentiation of $f_r(\nu)$ with respect to $\nu$ yields
\begin{align*}
\frac{f_r'(\nu)}{f_r(\nu)}= \Psi_{\mathtt{k}}(\nu+\mathtt{k})-
\Psi_{\mathtt{k}}(r\mathtt{k}+\nu+\mathtt{k})\leq 0,
\end{align*}
in view of the fact that the $\Psi_{\mathtt{k}}$ is increasing functions on $(-\mathtt{k}, \infty).$ This implies that $f_r(\nu)$ is decreasing.

Thus for $\mu \geq \nu >-\mathtt{k}$, it follows that
\begin{align*}
\sum_{r=0}^\infty f_r(\nu) x^{2r} \geq \sum_{r=0}^\infty f_r(\mu) x^{2r},
\end{align*}
which is equivalent to say that the function $\nu \mapsto \mathcal{I}^{\mathtt{k}}_{\nu}$ is decreasing on $(-\mathtt{k}, \infty)$ for some fixed $x >0$. 

The twice logarithmic differentiation of $f_r(\nu)$ yields
\begin{align*}
\frac{\partial^2}{\partial \nu^2}\bigg(\log(f_r(\nu)\bigg)
&=\Psi_k'(\nu+\mathtt{k})-\Psi_k'(r\mathtt{k}+\nu+\mathtt{k})\\
&=\sum_{n=0}^\infty  \left(\frac{1}{(n\mathtt{k}+\nu+\mathtt{k})^2}
- \frac{1}{(n\mathtt{k}+r\mathtt{k}+\nu+\mathtt{k})^2}\right)\\
&=\sum_{n=0}^\infty\frac{r \mathtt{k}(2n \mathtt{k}+ r \mathtt{k}+2 \nu+2\mathtt{k}) }{(n\mathtt{k}+\nu+\mathtt{k})^2(n\mathtt{k}+r\mathtt{k}+\nu+\mathtt{k})^2} \geq 0.
\end{align*}
for all $ \mathtt{k}>0$ and $\nu > -\mathtt{k}$. Thus $ \nu \mapsto f_r(\nu)$ is log-convex on $(-\mathtt{k}, \infty)$. In view of the fact that sums of log-convex functions are also log-convex, it follows that $\mathcal{I}_{\nu}^{\mathtt{k}}$ is log-convex on $(-\mathtt{k}, \infty)$ for each fixed $x>0$. \qedhere
\end{proof}

\begin{remark}
One of the most significance  consequences of the Theorem $\ref{finalThm}$  is the Tur$\acute{\text{a}}$n-type inequality for the function $\mathcal{I}_{\nu}^{\mathtt{k}}$. From the definition of log-convexity, it follows from Theorem $\ref{finalThm}$ (c) that
\begin{align*}
\mathcal{I}_{\alpha \nu_1+(1-\alpha)\nu_2}^{\mathtt{k}}(x) \leq \left(\mathcal{I}_{\nu_1}^{\mathtt{k}}\right)^{\alpha}(x)
\left(\mathcal{I}_{\nu_2}^{\mathtt{k}}\right)^{1-\alpha}(x),
\end{align*}
where $\alpha \in [0,1]$, $\nu_1, \nu_2 > -\mathtt{k}$, and $x >0$. Choose $\alpha=1/2$, and for any $\mathtt{a} \in \mathbb{R}$, let $\nu_1=\nu-\mathtt{a}$ and $\nu_2=\nu+\mathtt{a},$.  Then the  above inequality yields reverse Tur\'an type inequality
\begin{align}\label{eqn-Turan-1}
\left(\mathcal{I}_{\nu}^{\mathtt{k}}(x)\right)^2 - \mathcal{I}_{\nu-\mathtt{a}}^{\mathtt{k}}(x)
\mathcal{I}_{\nu+\mathtt{a}}^{\mathtt{k}}(x)
\leq 0
\end{align}
for any $\nu \geq |a|-\mathtt{k}$.
\end{remark}

Our final result is based on the Chebyshev integral inequality \cite[p. 40]{Mi}, which states the following: suppose $f$ and $g$  are two integrable functions and monotonic in the same sense (either both decreasing or both increasing). Let $q: (a, b) \to \mathbb{R}$ be a positive integrable function. Then
\begin{align}\label{eqn:chebyshev-1}
\left(\int_a^b q(t) f(t) dt\right) \left(\int_a^b q(t) g(t) dt\right) \leq \left(\int_a^b q(t) dt\right) \left(\int_a^b q(t) f(t) g(t) dt\right).
\end{align}
The inequality in \eqref{eqn:chebyshev-1} is reversed if $f$ and $g$ are monotonic but in the opposite sense.

Follwoing function is required
\begin{align}\label{normalized-k-2}
\mathcal{J}_\nu^{\mathtt{k}}(x):= \left(\frac{2}{x}\right)^{\frac{\nu}{\mathtt{k}}}\Gamma_{\mathtt{k}}
     (\nu+\mathtt{k}) \mathtt{J}_\nu^{\mathtt{k}}(x)=\sum_{r=0}^\infty g_r(\nu) x^{2r},
\end{align}
where,
\begin{align}\label{coef-2}
g_r(\nu)= \frac{(-1)^r\Gamma_{\mathtt{k}}{(\nu+\mathtt{k})}}{\Gamma_{\mathtt{k}}{(r{\mathtt{k}}+\nu+\mathtt{k})}4^r r!}.
\end{align}

\begin{theorem} Let $\mathtt{k}>0$. Then for $\nu \in (-3\mathtt{k}/4,- \mathtt{k}/2] \cup [\mathtt{k}/2, \infty) $
\begin{align}\label{chv-1}
\mathcal{I}_{\nu}^{\mathtt{k}}(x)\mathcal{I}_{\nu+\tfrac{\mathtt{k}}{2}}^{\mathtt{k}}(x)
\leq \frac{\sqrt{\mathtt{k}}}{x} \sin\left(\frac{x}{\mathtt{k}}\right)\mathcal{I}_{2\nu+\tfrac{\mathtt{k}}{2}}^{\mathtt{k}}(x).
\end{align}
and
\begin{align}\label{chv-2}
\mathcal{J}_{\nu}^{\mathtt{k}}(x)\mathcal{J}_{\nu+\tfrac{\mathtt{k}}{2}}^{\mathtt{k}}(x)
\leq \frac{\sqrt{\mathtt{k}}}{x} \sinh\left(\frac{x}{\mathtt{k}}\right)\mathcal{J}_{2\nu+\tfrac{\mathtt{k}}{2}}^{\mathtt{k}}(x).
\end{align}
The inequalities in \eqref{chv-1} and \eqref{chv-2} are reversed if $\nu \in (-\mathtt{k}/2, \mathtt{k}/2)$.
\end{theorem}
\begin{proof}
Define the functions $p$, $f$ and $g$ on $[0, 1]$ as
\[ q(t)= \cos\left(\tfrac{x t}{\sqrt{\mathtt{k}}}\right), \quad  f(t)=(1-t^2)^{\tfrac{v}{\mathtt{k}}-\tfrac{1}{2}}, \quad   g(t)=(1-t^2)^{\tfrac{v}{\mathtt{k}}+\tfrac{1}{2}}.\]
Then
\begin{align*}
\int_0^1 q(t) dt&= \int_0^1  \cos\left(\tfrac{x t}{\sqrt{\mathtt{k}}}\right) dt= \frac{\sqrt{\mathtt{k}}}{x} \sin\left(\tfrac{x }{\sqrt{\mathtt{k}}}\right);\\
\int_0^1 q(t) f(t)dt&= \int_0^1  \cos\left(\tfrac{x t}{\sqrt{\mathtt{k}}}\right)(1-t^2)^{\tfrac{v}{\mathtt{k}}-\tfrac{1}{2}} dt=
\mathcal{I}_{\nu}^{\mathtt{k}}(x) ,\quad  \text{if}\quad  \nu \geq - \mathtt{k}; \\
\int_0^1 q(t) g(t)dt&= \int_0^1  \cos\left(\tfrac{x t}{\sqrt{\mathtt{k}}}\right)(1-t^2)^{\tfrac{v}{\mathtt{k}}+\tfrac{1}{2}} dt=
\mathcal{I}_{\nu+\mathtt{k}}^{\mathtt{k}}(x) ,\quad  \text{if} \quad  \nu \geq - 2\mathtt{k}; \\
\int_0^1 q(t) f(t) g(t)dt&= \int_0^1  \cos\left(\tfrac{x t}{\sqrt{\mathtt{k}}}\right)(1-t^2)^{\tfrac{2 v}{\mathtt{k}}} dt=
\mathcal{I}_{2 \nu+\tfrac{\mathtt{k}}{2}}^{\mathtt{k}}(x) ,\quad  \text{if} \quad  \nu \geq - \tfrac{3\mathtt{k}}{4}; \\
\end{align*}
Since the functions $f$ and $g$ both are decreasing for $\nu \geq \mathtt{k}/2$ and both are  increasing for  $\nu \in (-3\mathtt{k}/4,- \mathtt{k}/2]$, the inequality \eqref{eqn:chebyshev-1} yields \eqref{chv-1}. On the other hand if $\nu \in (-\mathtt{k}/2, \mathtt{k}/2)$, the function $f$ is increasing but $g$ is decreasing, and hence the inequality in \eqref{chv-1} is reversed.

Similarly, the inequality in \eqref{chv-2} can be derived by using \eqref{eqn:chebyshev-1} by choosing
\[ q(t)= \cosh\left(\tfrac{x t}{\sqrt{\mathtt{k}}}\right), \quad  f(t)=(1-t^2)^{\tfrac{v}{\mathtt{k}}-\tfrac{1}{2}}, \quad   g(t)=(1-t^2)^{\tfrac{v}{\mathtt{k}}+\tfrac{1}{2}}.\]

\end{proof}

\end{document}